\documentclass[runningheads]{llncs}
\usepackage{enumerate}
\usepackage{paralist}
\usepackage{graphicx}
\usepackage{amsmath,amssymb}
\usepackage{comment}
\usepackage{svg}
\usepackage{amsfonts}
\newcommand{\N}{\mathbb{N}}
\newcommand{\Z}{\mathbb{Z}}
\newcommand{\lex}{\text{lex}}

\renewcommand\figurename{Figure}

\begin{document}
\title{Generalized Lyndon Factorizations of Infinite Words}
\author{Amanda Burcroff\orcidID{0000-0002-0032-4190} \and
Eric Winsor\orcidID{0000-0003-1922-4648} }
\authorrunning{A. Burcroff and E. Winsor}

\institute{University of Michigan, Ann Arbor MI 48109, USA \\
\email{\{burcroff,rcwnsr\}@umich.edu}}
\maketitle              % typeset the header of the contribution
\begin{abstract}
A generalized lexicographic order on words is a lexicographic order where the total order of the alphabet depends on the position of the comparison.  A generalized Lyndon word is a finite word which is strictly smallest among its class of rotations with respect to a generalized lexicographic order.  This notion can be extended to infinite words: an infinite generalized Lyndon word is an infinite word which is strictly smallest among its class of suffixes.  We prove a conjecture of Dolce, Restivo, and Reutenauer: every infinite word has a unique nonincreasing factorization into finite and infinite generalized Lyndon words.  When this factorization has finitely many terms, we characterize the last term of the factorization.  Our methods also show that the infinite generalized Lyndon words are precisely the words with infinitely many generalized Lyndon prefixes.

\end{abstract}
\section{Introduction}\label{1}
A {\it rotation} of a finite word $w$ is a word of the form $vu$, where $w = uv$ is a factorization of $w$.  A finite word is called {\it Lyndon} if it is strictly smallest among its class of rotations with respect to the standard lexicographic order. In particular, every finite word is a conjugate of some power of a Lyndon word. Lyndon words were introduced in 1953 by Shirshov in \cite{Shi} and studied by Lyndon in \cite{Lyn}.  Lyndon words have been given various names throughout their history, including {\it standard lexicographic sequences}, {\it regular words}, and {\it prime words}.  These names hint at their significant role in the factorization of words.  

Let $A^*$ denote the free monoid on a totally ordered (possibly infinite) alphabet $A$, where $A^*$ is ordered lexicographically.  The Chen-Fox-Lyndon factorization theorem for words states that the Lyndon words form a basis for $A^*$ \cite{CFL}.  Put more concretely, any finite word on $A$ can be written uniquely as a product of nonincreasing Lyndon words.  

About 40 years later, infinite Lyndon words were introduced in \cite{SMDS}.  There are several equivalent definitions, but we use the definition which focuses on the idea of rotation.  An infinite word is called {\it Lyndon} if it is strictly smallest among its suffixes with respect to the standard lexicographic order.  If $w$ is an infinite word with a nontrivial factorization $uv$, the suffix $v$ can be viewed as the rotation with respect to this factorization.  Let $A^\omega$ denote the set of sequences, or infinite words, over $A$.  These too yielded deep factorization properties; Siromoney et al. showed that every sequence in $A^\omega$ has a unique factorization as a nonincreasing product of finite and infinite Lyndon words.

The extension of the Lyndon property to generalized lexicographic orders came about 10 years later by Reutenauer \cite{Reu}.  A {\it generalized lexicographic order} is a modified lexicographic order where the total order of the alphabet depends on the index of comparison. This naturally induces a notion of finite and infinite {\it generalized Lyndon words} under a generalized lexicographic order.  (See Section \ref{2}.) Reutenauer showed that the finite generalized Lyndon words form a basis for $A^*$ using Hall set theory, and Dolce et al. provided a combinatorial proof in 2018 \cite{Reu,DRR}.  Generalized Lyndon words are studied further by Dolce et al. in \cite{DRR2}.

An example of a generalized lexicographic order is the {\it alternating order} $\leq_{\text{alt}}$, where the alphabet is given its standard order when the index of comparison is odd and its opposite order when the index is even.  This order can be connected with continued fractions by noting that the map $\phi: \N^\omega \to \mathbb{R}$ defined by
\vspace{-2mm}$$\displaystyle\phi(x_1x_2\cdots) = x_1 +  \cfrac{1}{x_2 + \cfrac{1}{\;\;\;\;\ddots}}\vspace{-1mm}$$
satisfies $u \leq_{\text{alt}} v$ in $\N^\omega$ if and only if $\phi(u) \leq \phi(v)$ in $\mathbb{R}$.  Generalized Lyndon words with respect to the alternating order are called {\it Galois words}, and Galois factorizations were given further characterization in \cite{DRR2}.  Another special case are the anti-Lyndon words, introduced in \cite{GM}, which are generalized Lyndon words with respect to the opposite lexicographic order.

Dolce et al. conjectured that the finite and infinite generalized Lyndon words provide a unique nonincreasing (with respect to $\omega$-powers) factorization of all infinite words.  Our main result is to show that this is indeed the case.  

In Section \ref{3}, we focus on words with a generalized Lyndon suffix.  Theorem \ref{L} shows that these are precisely the words with finitely many terms in their nonincreasing generalized Lyndon factorization.  Moreover, we characterize the last term as the first generalized Lyndon suffix (with respect to the index).  

Sections \ref{4} and \ref{5} focus on the existence and uniqueness, respectively, of nonincreasing generalized Lyndon factorizations for words which have no generalized Lyndon suffix.  In the process we develop powerful machinery to take advantage of the strong properties of these factorizations.  A product of this machinery is presented briefly in Section \ref{6}, where we show that an infinite word is generalized Lyndon if and only if it has infinitely many generalized Lyndon prefixes. This is the generalized analogue of the result of Siromoney et al. showing that infinite Lyndon words are precisely the words with infinitely many Lyndon prefixes.  

\section{Preliminaries}\label{2}
Let $\N = \{1,2,3,\dots\}$.  Words are finite or infinite (to the right) sequences of letters from a fixed (possibly infinite) alphabet $A$.  For $i < j$, the contiguous substring beginning at the $i^{\mathrm{th}}$ letter and ending with the $j^{\mathrm{th}}$ (inclusive on both ends) is denoted $x[i,j]$.  A word $v$ is a {\it factor} of $x$ if $x = uvw$ for (possibly empty) words $u$ and $w$.  In the case that $u$ is empty, $v$ is a {\it prefix} of $x$, and if $w$ is empty, then $v$ is a {\it suffix} of $x$.  If in addition $w$ (resp. $u$) is nonempty, we say that the prefix (resp. suffix) is {\it proper}.  If $x$ is an infinite word, the suffix of $x$ beginning at the $j^{\mathrm{th}}$ index of $x$ is denoted $x[j,\infty)$.  The length of a finite word $w$ is denoted by $|w|$.

Let $A^\infty = A^* \cup A^\omega$.  Given a total order on an alphabet $A$, the {\it lexicographic ordering} $<_{\lex}$ on $A^\infty$ is defined such that $x <_{\lex} y$ if and only if $x$ is a proper prefix of $y$ or $x = pas$ and $y = pbs'$ for words $p,s,s'$ and letters $a < b$.  We are primarily interested in a generalization of this order.

For each $n \in \N$, let $<_n$ be a total order on $A$.  The {\it generalized lexicographic order} $<$ induced by $(<_n)_{n \in \N}$ is defined such that $x < y$ if and only if $x$ is a proper prefix of $y$ or $x = pas$ and $y = pbs'$ for words $p,s,s'$ and letters $a <_{|p| + 1} b$.

If $u$ is a prefix of $v$ or $v$ is a prefix of $u$, we write $u \sim v$.  Note that if $|u| = |v|$, then $u \sim v$ implies $u = v$.  We will use the $\sim$ operator ``transitively'', where the expression $w_1 \sim w_2 \sim \cdots \sim w_n$ implies that the shortest of the $n$ words is a prefix of the rest.  We also define a modified comparison operator $\lesssim$ such that $w_1 \lesssim \cdots \lesssim w_n$ if the prefixes $p_i$ of $w_i$ having length $\min\{|w_1|,\dots,|w_n|\}$ satisfy $p_1 \leq \cdots \leq p_n$, where $\leq$ is the generalized lexicographic order.  The same property of only comparing the prefixes up to the length of the shortest word in a chain also applies when the operators $\sim$ and $\lesssim$ are applied together in a chain.

A finite word $v$ is called a {\it power} of a finite word $u$ if $v = u^k$ for some integer $k \geq 2$.  Let the $\omega$-power of $u$, denoted by $u^\omega$, be the infinite word $\prod_{i=1}^\infty u$.  An infinite word $v$ is called a {\it power} of a finite word $u$ if $v = u^\omega$; we also say that $v$ is {\it periodic}.  If $u$ is infinite, we use the convention $u^\omega = u$. An infinite word with a periodic suffix is called {\it eventually periodic}, and an infinite word which is not eventually periodic is called {\it aperiodic}.  A word which is not a power is called {\it primitive}.  A finite word $w$ is called a {\it fractional power} of a finite word $u$ if $w \sim u^\omega$.  We write $w = u^{|w|/|u|}$, e.g., $01 = (0111)^{1/2}$. See \cite{Lot2}, \cite{Lot}, and \cite{PR} for more on the combinatorics of words.

A word $w$ is a {\it finite generalized Lyndon word} if it is strictly smallest among its class of rotations with respect to a generalized lexicographic order.  That is, for any nontrivial factorization $w = uv$, we have $uv < vu$.  An infinite word $w$ is an {\it infinite generalized Lyndon word} if it is strictly smallest among its class of suffixes with respect to a generalized lexicographic order.  A nonincreasing generalized Lyndon factorization of a word $w$ is a product of the form $w = \prod_{i = 1}^n \ell_i$  where $n \in \N \cup \{\infty\}$, each $\ell_i$ is generalized Lyndon, and $\ell_i^\omega \geq \ell_{i+1}^\omega$ for all $i \in [1,n)$.

\section{Existence and Uniqueness of Finite Factorizations}\label{3}

In this section, we show that the words admitting a unique finite nonincreasing generalized Lyndon factorization are precisely the words that have a generalized Lyndon suffix. 

\begin{lemma}\emph{(\cite{DRR}, Lemma 31)} \label{C1}
Let $u,v$ be nonempty finite words.  Then the following four conditions are equivalent:
\\\emph{(1) $u^\omega < v^\omega$ \hspace{.75cm} (2) $(uv)^\omega < v^\omega$ \hspace{.75cm} (3) $u^\omega < (vu)^\omega$ \hspace{.75cm} (4) $(uv)^\omega < (vu)^\omega$.}
\end{lemma}

We will also make use of a result by Lyndon and Sch{\" u}tzenberger \cite{LS} concerning commuting words, which can easily be strengthened when one of the words is generalized Lyndon.

\begin{lemma}\emph{(\cite{LS})}\label{C2}
Two finite words commute if and only if they are powers of a common word.
\end{lemma}

\begin{corollary}\label{I}
Suppose $u$ is a finite generalized Lyndon word, $v$ is any finite word, and $uv = vu$.  Then $v$ is a power of $u$.
\end{corollary}
\begin{proof}
This follows from Lemma \ref{C2} and the fact that generalized Lyndon words are primitive.
\end{proof}

\begin{lemma}\label{J}
Suppose $u$ and $v$ are finite words satisfying $u^\omega \lesssim v \lesssim u^nv$ (resp. $u^\omega \gtrsim v \gtrsim u^nv$) for some $n \in \N$.  Then $v\sim u^\omega$.  
\end{lemma}
\begin{proof}
Suppose there exists a maximum nonnegative integer, $m$, such that $u^m \sim v$; note that $|u^m| < |v|$.  Then
$$u^{m+1} \leq u^\omega  \lesssim v \lesssim u^n v \sim u^{m + n}.$$
Thus $v \sim u^{m+n}$, a contradiction to our choice of $m$.  The proof proceeds analogously for the case where the inequalities are reversed.
\end{proof}

\begin{theorem}\label{G}
Suppose $w$ is an infinite word. If $w$ is a nonincreasing product of finite generalized Lyndon words, then $w$ has no generalized Lyndon suffixes.
\end{theorem}
%
% the environments 'definition', 'lemma', 'proposition', 'corollary',
% 'remark', and 'example' are defined in the LLNCS documentclass as well.
%
\begin{proof}
Suppose $w$ has a generalized Lyndon suffix $\ell$.  Without loss of generality, we can assume $w = \ell_0\ell_1\ell_2\cdots$ and $\ell = u\ell_1\ell_2\cdots$, where each $\ell_i$ is a generalized Lyndon word, $\ell_0^\omega \geq \ell_1^\omega \geq \cdots$, and $u$ is a suffix of $\ell_0$.  Since $\ell_0$ is generalized Lyndon, Lemma \ref{C1} implies $u^\omega \geq \ell_0^\omega$. Furthermore, since $\ell$ is generalized Lyndon, we have $\ell_r \gtrsim u$ for all $r \in \N$.  Thus, for all $r \in \N$ we have $u^\omega \geq \ell_0^\omega \geq \ell_r^\omega \sim \ell_r \gtrsim u$, hence $\ell_r \sim u$.

Suppose that there exists some $r \in \N$ such that $|\ell_r| < |u|$.  Note that each such $\ell_r$ is a prefix of $u$.  By the nonincreasing property of the generalized Lyndon factors, either there exist finitely many such $r$, or there exists some $n \in \N$ and $\alpha \in (0,1)$ such that for all $r \geq n$, we have $\ell_r = u^\alpha$. The latter case holds because there are only finitely many prefixes of $u$, so one prefix must appear infinitely many times. By the nonincreasing property of the factorization, this means that all terms in the factorization after the first term equal to this prefix must also equal this prefix. Observe that in the latter case, we have
$$(u^\alpha)^\omega = \ell_n^\omega \leq u^\omega \sim u \lesssim \ell_n\ell_{n+1}\cdots = (u^\alpha)^\omega,$$
hence $u \sim (u^\alpha)^\omega$.  

We conclude that there exists a minimal $k \in \N$ such that $\ell_k = u^\alpha$ for some $\alpha \in (0,1)$ and $u \sim \ell_{k+1}\ell_{k+2}\cdots$.  Since $u^\alpha u \sim \ell_k\ell_{k+1}\cdots$, then $u^\alpha u \gtrsim u\ell_1$. Thus, $u^\alpha u\gtrsim u\sim u^\omega\geq \ell_k^\omega=(u^\alpha)^\omega$, so Lemma \ref{J} implies $u\sim (u^\alpha)^\omega$. Suppose $|\ell_1|\geq |u|$. Then $uu^\alpha$ is a prefix of $w$, so
$$(u^\alpha)^\omega \sim u^\alpha u \geq u u^\alpha \sim u^\omega \geq (u^\alpha)^\omega,$$
hence $u$ is a power of $u^\alpha$ by Corollary \ref{I}.  Thus, $u$ is not generalized Lyndon, so $u^\omega > \ell_0^\omega \geq \ell_k^\omega = (u^\alpha)^\omega = u^\omega$, a contradiction.  

Thus, we must have that $|\ell_1| < |u|$.  By the minimality of $k$, we have that $|\ell_r| < |u|$ for $1 \leq r \leq k$, which implies that $\ell_r \sim u \sim (u^\alpha)^\omega = \ell_k^\omega$ for $1 \leq r \leq k$. As $u \sim \ell_k^\omega$, we have that $u \sim \ell_k u \sim \ell_k\ell_{k+1}$. Hence
$$\ell_{k-1}^\omega \leq u^\omega \sim u \lesssim \ell_{k-1}\ell_k\ell_{k+1}\cdots \sim \ell_{k-1}u.$$
In particular, by Lemma \ref{J}, we have $u \sim \ell_{k-1}^\omega$ and $u \sim \ell_{k-1}\ell_k\cdots$.  We repeat this process, showing that $u\sim \ell_i^\omega$ and $u \sim \ell_i\ell_{i+1}\cdots$ for all $1 \leq i \leq k$.  Hence $\ell_1 u \sim \ell_1\ell_2\cdots \gtrsim u\ell_1$.  However, since $\ell_1^\omega \leq u^\omega$, Lemma \ref{C1} implies $\ell_1 u \leq u\ell_1$.  Thus $u$ and $\ell_1$ commute, so Corollary \ref{I} implies $u$ is a power of $\ell_1$.  In particular, $\ell_1$ is a proper suffix of $\ell_0$, so $\ell_1^\omega > \ell_0^\omega$, contradicting our nonincreasing assumption.

Thus, we must have $|\ell_r| \geq |u|$ for all $r \in \N$.  $w$ has a generalized Lyndon suffix, so it cannot be periodic. We can fix $s$ to be the smallest index such that $\ell_s \not= u$.  By Lemma \ref{J}, the inequality
$$u^\omega \geq \ell_s^\omega \sim \ell_s \gtrsim u \ell_1 \cdots \ell_s = u^s \ell_s,$$
implies that $\ell_s \sim u^\omega$.  Hence $\ell_s = u^{n + \beta}$ for some $n \in \N$ and $\beta \in [0,1)$.  On the one hand, we have $u^\omega \geq \ell_s^\omega = (u^{n+\beta})^\omega$, hence $u^nuu^\beta \geq u^nu^\beta u$.  On the other hand, since $u^nu u^\beta \sim u^{s + n + \beta}$ is a prefix of $\ell$ and $u^nu^\beta u\sim\ell_s\ell_{s+1}$ is a factor, we have $u^nu u^\beta \leq u^nu^\beta u$ because $\ell$ is generalized Lyndon.  Hence, Lemma \ref{C2} implies $(u^\beta)^\omega = u^\omega = \ell_s^\omega$, contradicting that $(u^\beta)^\omega > \ell_s^\omega$ by the generalized Lyndon property of $\ell_s$.
\end{proof}

\begin{lemma}\label{H}
If $u$ is a finite word and $v$ is an infinite word, then $u^\omega > v$ (resp. $u^\omega < v$) if and only if $uv > v$ (resp. $uv < v$).
\end{lemma}
\begin{proof}
Suppose $u^\omega > v$.  Let $j$ be the largest integer such that $u^j \sim v$.  Hence $v = u^jv'$ for some infinite word $v' \not\sim u$.  Thus, the comparison between $uv$ and $v$ happens between index  $j|u| + 1$ and index $(j+1)|u|$, inclusive.  In particular, $uv \sim u^{j+1} > v$.

Now suppose $uv > v$.  Let $k$ be the largest index such that $u^k \sim v$.  Thus, the comparison between $uv$ and $v$ happens between index  $k|u| + 1$ and index $(k+1)|u|$, inclusive.  In particular, $u^\omega \sim u^{k+1} \sim uv > v$.  

The proof with the reverse inequalities proceeds analogously.
\end{proof}

\begin{comment}
{\begin{corollary}\label{K}
An infinite generalized Lyndon word has a unique nonincreasing factorization into generalized Lyndon words.
\end{corollary}
\begin{proof}
Fix an infinite generalized Lyndon word $w$.  From Theorem \ref{G}, we know that $w$ cannot have a nonincreasing factorization into infinitely many generalized Lyndon words.  Suppose $w = \ell_1\cdots \ell_n \ell$, where $\ell_1,\dots,\ell_n,\ell$ are generalized Lyndon words satisfying $\ell_1^\omega \geq \cdots \geq \ell_n^\omega \geq \ell$.  A repeated application of Lemma \ref{C1} yields that $(\ell_1\cdots \ell_n)^\omega \geq \ell_n^\omega \geq \ell$.  By Lemma \ref{H}, we have $\ell_1\cdots\ell_n\ell \geq \ell$, which contradicts the fact that $w$ is a generalized Lyndon word.  Thus, $w$ is its own unique nonincreasing factorization into generalized Lyndon words.  
\end{proof}
\end{comment}

In order to show the existence and uniqueness of generalized Lyndon factorizations of infinite words, we will invoke a theorem of Reutenauer which gives the analogous result for finite words \cite{Reu}. 

\begin{theorem}\emph{\cite{Reu,DRR}}\label{C3}
Any finite word has a unique nonincreasing factorization into generalized Lyndon words.
\end{theorem}

\begin{theorem}\label{L}
An infinite word with an infinite generalized Lyndon suffix has a unique factorization into generalized Lyndon words, and this factorization is finite.  Furthermore, the last term in this factorization is the first generalized Lyndon suffix by index.
\end{theorem}
\begin{proof}
We first show existence.  Let $\ell$ be the first generalized Lyndon suffix of $w$ by index, that is, $w = v\ell$ where the length of $v$ is minimum such that $\ell$ is generalized Lyndon.  Let $\ell_1, \dots, \ell_n$ be the unique nonincreasing factorization of $v$ from Theorem \ref{C3}.  It is enough to show that $\ell_n^\omega \geq \ell$, as this will yield $\ell_1, \dots, \ell_n, \ell$ as a nonincreasing generalized Lyndon factorization of $w$.

Suppose that $\ell_n^\omega < \ell$.  By Lemma \ref{H}, this implies $\ell_n\ell < \ell$.  Let $s$ be the shortest (not necessarily proper) suffix of $\ell_n$ such that $s\ell$ is minimal.  Note that we have $s\ell \leq \ell_n\ell < \ell$, so $s$ is nonempty.  However, by construction we have $s\ell \leq s'\ell$ for every suffix $s'$ of $s$. Notably, $s\ell\leq \ell\leq\ell'$ for any suffix $\ell'$ of $\ell$ because $\ell$ is generalized Lyndon. Thus $s\ell$ is generalized Lyndon.  This contradicts our choice of $\ell$ to be the first generalized Lyndon suffix of $w$. Therefore $\ell_n^\omega \geq \ell$, so we have produced a nonincreasing factorization of $w$.

By Theorem \ref{G}, any factorization of $w$ must have only finitely many terms. Let $\ell_1, \dots, \ell_n \ell$ be a nonincreasing factorization of $w$ into generalized Lyndon words.  Suppose, seeking a contradiction, that $\ell$ is not the longest generalized Lyndon suffix of $w$, i.e., there is a suffix $s$ of $w$ of the form $u\ell_{j+1}\cdots\ell_n\ell$ where $u$ is a suffix of $\ell_j$.  From the nonincreasing property of the factorization $w$ and the generalized Lyndon property of $\ell_j$, we know $u^\omega \geq \ell_j^\omega \geq \cdots \geq \ell_n^\omega \geq \ell$.  By Lemma \ref{C1}, $(u\ell_j)^\omega \geq \ell_j^\omega$.  Inductively, we find $(u\ell_j\cdots\ell_n)^\omega \geq \ell_n^\omega \geq \ell$.  Thus,  by Lemma \ref{H}, we have $s = u\ell_j\cdots\ell_n\ell \geq \ell$, contradicting that $s$ is generalized Lyndon.  

Now that we have uniquely determined $\ell$, the other factors $\ell_1,\dots,\ell_n$ are uniquely determined.  This follows because the prefix $w[1, |\ell_1|+\cdots + |\ell_n|]$ of $w$ has a unique nonincreasing factorization into generalized Lyndon words.  Thus by our initial assumption that $\ell_1, \dots, \ell_n, \ell$ is a nonincreasing factorization of $w$, the unique factorization of $w[1, |\ell_1|+\cdots + |\ell_n|]$ must be $\ell_1, \dots, \ell_n$.
\end{proof}
\section{Existence of Infinite Factorizations}\label{4}

In this section, we describe a method to construct an infinite factorization of a word with no generalized Lyndon suffix by taking a limit of the finite factorizations of some of its prefixes.

\begin{lemma}\label{A}
If a primitive infinite word has infinitely many generalized Lyndon prefixes, then it is a generalized Lyndon word.
\end{lemma}
\begin{proof}
Let $w$ be a primitive word which is not infinite generalized Lyndon, and let $m\in \N$ be minimal such that $w[m,\infty) < w$.  Let $i$ be the index of comparison between $w[m,\infty)$ and $w$.  Then for any $n \geq m + i$, we have $w[m,n] \sim w[m,\infty) < w \sim w[1,n]$ with a comparison at index $i$.  Thus $w[1,n]$ is not generalized Lyndon for any $n \geq m+i$, so we can conclude that $w$ has finitely many generalized Lyndon prefixes.
\end{proof}

\begin{lemma}\label{M}
If $\ell$ is a finite word that is not generalized Lyndon, then $\ell^\omega$ has finitely many generalized Lyndon prefixes.  
\end{lemma}
\begin{proof}
If $\ell$ is not generalized Lyndon, then we can write $\ell = uv$ where $vu < uv$ for some prefix $u$.  Observe that $vu$ will be a factor of any prefix of $\ell^\omega$ having length at least $|\ell| + |u|$, and $uv$ will be a prefix of any such prefix of $\ell^\omega$.  Thus any prefix of $\ell^\omega$ having length at least $|\ell| + |u|$ is not generalized Lyndon.
\end{proof}

\begin{theorem}\label{N}
An infinite word has a nonincreasing factorization into generalized Lyndon words.
\end{theorem}
\begin{proof} 
Fix an infinite word $w$.  Theorem \ref{L} completes the proof in the case that $w$ has an infinite generalized Lyndon suffix.  So we can assume that $w$ has no infinite generalized Lyndon suffix.  In particular, $w$ is not generalized Lyndon.

We will first consider the case where $w$ is not eventually periodic.  Since $w$ is not generalized Lyndon, Lemma \ref{A} implies that $w$ has finitely many generalized Lyndon prefixes.   Thus one of its generalized Lyndon prefixes must appear in the factorization of $w[1,n]$ yielded by Theorem \ref{C3} for infinitely many $n \in \N$.  Let $\ell_1$ be such a prefix, and let $w = \ell_1w_1$.  

We will now inductively construct a factorization of $w$.  Suppose we can write $w = \ell_1 \cdots \ell_k w_k$ such that each $\ell_j$ is a finite generalized Lyndon word, $\ell_1^\omega \geq \cdots \geq \ell_k^\omega$, and $w$ has infinitely many prefixes whose factorizations begin with $\ell_1, \dots, \ell_k$.  Since $w$ has no generalized Lyndon suffixes, $w_k$ is not generalized Lyndon, so it must have finitely many generalized Lyndon prefixes.  Since infinitely many prefixes of $w$ have factorizations beginning with $\ell_1, \dots, \ell_k$, one of the generalized Lyndon prefixes of $w_k$, which we label $\ell_{k+1}$, must be such that infinitely many prefixes of $w$ have factorizations beginning with $\ell_1,\dots, \ell_k, \ell_{k+1}$.  We can then write $w = \ell_1\cdots \ell_{k+1}w_{k+1}$.  Note that by construction, $\ell_k^\omega \geq \ell_{k+1}^\omega$.  By induction, we get a nonincreasing generalized Lyndon factorization $\ell_1,\ell_2,\dots$ of $w$.

Now suppose that $w$ is eventually periodic.  If $w$ is a power of a generalized Lyndon word $\ell$, we can use the factorization $w = \ell^\omega$.  Otherwise, $w$ is a power of a finite word that is not generalized Lyndon or $w$ is primitive.  In either case, Lemmas \ref{A} and \ref{M} imply that $w$ has finitely many generalized Lyndon prefixes.  We can thus apply the construction from the previous paragraph, in each step yielding a factorization of $w$ starting with $\ell_1,\cdots, \ell_k$.  This process will halt only if $w_k$ has infinitely many generalized Lyndon prefixes $p_i$ such that $\ell_1, \dots, \ell_k, p_i$ is the factorization of a prefix of $w$.  By Lemmas \ref{A} and \ref{M}, this implies that $w_k$ is a power of a generalized Lyndon word $\ell$.  Moreover, since the $p_i$'s have unbounded length and  $\ell_k^\omega \geq p_i^\omega$, we must have $\ell_k^\omega \geq \ell^\omega$.  Therefore $\ell_1, \dots, \ell_k, \ell^\omega$ is a factorization of $w$.  Thus, in any case, this construction yields a nonincreasing factorization of $w$ into generalized Lyndon words.
\end{proof}

\section{Uniqueness of Infinite Factorizations}\label{5}

We will determine the uniqueness of the factorization constructed in Section \ref{4}, handling first the eventually periodic words and then aperiodic words with no generalized Lyndon suffix. 

\begin{theorem}\label{O}
An eventually periodic infinite word has a unique nonincreasing factorization into generalized Lyndon words.
\end{theorem}
\begin{proof}
Fix an infinite word $w$ with a periodic suffix.  Observe that this implies we can write $w$ as $u\ell^\omega$ where $u$ is a (possibly empty) finite word and $\ell$ is a nonempty finite generalized Lyndon word. We may assume $w$ has no generalized Lyndon suffix, as this case is handled by Theorem \ref{L}.  

We first claim that the factorization (from Theorem \ref{N}) of $w = \prod_{i=1}^\infty \ell_i$ must terminate with $\ell^\omega$. Since $\ell$ is generalized Lyndon hence not equal to any of its rotations, we have that $\ell^\omega[i,\infty) = \ell^\omega$ if and only if $i-1$ is an integer multiple of $|\ell|$.  Moreover, if $i-1$ is not a multiple of $|\ell|$, then $\ell^\omega[i,\infty)$ is a power of a word which is not generalized Lyndon and hence has finitely many generalized Lyndon prefixes by Lemma \ref{M}.  If one of these generalized Lyndon prefixes, $\ell'$, appears infinitely many times in the factorization of $w$, then $(\ell')^\omega$ is a suffix of $w$. Since $(\ell')^\omega$ and $\ell^\omega$ are suffixes of $w$, they are powers of rotations of $\ell$ and $\ell'$, respectively. Because $\ell$ and $\ell'$ are generalized Lyndon, this means that $\ell=\ell'$.  That is, only finitely many terms of the factorization are not equal to $\ell$.  Thus, we can conclude $\ell_i = \ell$ for sufficiently large $i$.

Now suppose $\ell_1, \dots, \ell_n, \ell^\omega$ and $h_1, \dots,  h_m, \ell^\omega$ are two distinct factorizations of $w$.  Note that $|\ell_1\cdots\ell_n| - |h_1\cdots h_m|$ must be an integer multiple of $|\ell|$, as $\ell$ is a generalized Lyndon word and hence not equal to any of its rotations.  Without loss of generality, assume $|\ell_1\cdots\ell_n| - |h_1\cdots h_m|>0$.  In this case, there exists $k \in \N$ such that $\ell_1\cdots \ell_n = h_1\cdots h_m \ell^k$, which violates the uniqueness of the nonincreasing generalized Lyndon factorization for finite words from Theorem \ref{C3}.  Thus, the nonincreasing factorization of $w$ into generalized Lyndon words is unique.
\end{proof}

\begin{lemma}\label{U}
Let $w = v\ell_1\ell_2\cdots \ell_n u$ be a finite generalized Lyndon word where $n \in \Z_{\geq 0}$, $\ell_i$ is a finite generalized Lyndon word for all $i\in \{1,\ldots,n\}$, $v$ is a suffix of a finite generalized Lyndon word $\ell_0$, $u$ is a prefix of a finite generalized Lyndon word $\ell_{n+1}$, and $\ell_0^\omega \geq \ell_1^\omega \geq \cdots \geq \ell_{n+1}^\omega$. Then $u \sim v$ and $u\sim v \sim \ell_i\cdots \ell_n u$ for all $i \in \{1,\dots,n\}$.
\end{lemma}
\begin{proof}
The generalized Lyndon property of $w$ implies $u \gtrsim v$.  The nonincreasing property of the factors implies $v \sim v^\omega \geq \ell_0^\omega \geq \ell_{n+1}^\omega \sim u$. Combining these inequalities, we have $u \sim v$.

Suppose $|u| \leq |v|$.  The generalized Lyndon property of $w$ and the nonincreasing property furthermore implies
$$\ell_n u \gtrsim v \sim u \sim v \sim v^\omega \geq \ell_0^\omega \geq \ell_n^\omega.$$
Hence Lemma \ref{J} implies that $u \sim \ell_n^\omega$, so $\ell_n u \sim u$.  Repeating this process, we can conclude $u \sim v \sim \ell_n u \sim \ell_{n-1}\ell_n u \sim \cdots \sim  \ell_1\cdots \ell_n u$.

Similarly, suppose $|u| > |v|$.  The generalized Lyndon property of $w$ and the nonincreasing property  implies
$$\ell_n v \sim \ell_n u \gtrsim v \sim v^\omega \geq \ell_0^\omega \geq \ell_n^\omega.$$
Hence Lemma \ref{J} implies that $v \sim \ell_n^\omega$, so $v \sim \ell_n v \sim \ell_n u$.  Repeating this process, we can conclude $v \sim u \sim  \ell_n u \sim \ell_{n-1}\ell_n u \sim \cdots \sim \ell_1\cdots \ell_n u$.
\end{proof}

\begin{lemma}\label{Q}
Let $w$ satisfy the hypotheses of Lemma \ref{U}. If $|u| \geq |v|$, then there exists some $m$ with $0\leq m \leq n$ such that
$$
\ell_j = 
\begin{cases}
v & \text{ if } 1\leq j \leq m\\
v^{\alpha_j} \text{ for some } \alpha_j \in (0,1) & \text{ if } m < j \leq n.
\end{cases}
$$ 
\end{lemma}
\begin{proof}
We assume that $\ell_j = v$ for $1\leq j\leq k$ with $0\leq k\leq n-1$ and proceed by induction on $k$. Note that the base case of $k=0$ is automatic. Furthermore, we suppose there exists $m$ with $0\leq m\leq n$ that satisfies the property of Lemma 8 when we restrict to considering $\ell_j$ with $j\leq k$. Note that we can have $m \geq k$. By Lemma \ref{U}, we have $\ell_{k+1} \sim u \sim v$ and $|v| \leq |u|$, hence $\ell_{k+1} \sim v$. Thus, if $|\ell_{k+1}| < |v|$, then we are done.

Suppose $m < k$ and $|\ell_{k+1}| \geq |v|$, so $v$ is a prefix of $\ell_{k+1}$.  Let $\ell_{m+i} = v^{\alpha_{i}}$ for $1 \leq i \leq k-m$, where each $\alpha_{i} \in (0,1)$.  Let $t = k-m$.  Thus
$v\ell_1\cdots\ell_k \ell_{k+1} \sim v^k v^{\alpha_1}\cdots v^{\alpha_{t}} v$.  By Lemma \ref{U}, we have $v \sim u \sim v^{\alpha_2}\cdots v^{\alpha_{k-m}} v$, so $v^{\alpha_1} v$ is a factor of $w$.  Since $vv^{\alpha_1}$ is a prefix of $w$, by the generalized Lyndon property we have $v^{\alpha_1}v \geq vv^{\alpha_1}$.  On the other hand, we have $v^\omega \geq \ell_0^\omega \geq \ell_{m+1}^\omega = (v^{\alpha_1})^\omega$, 
 which implies $v^\omega \geq (v^{\alpha_1} v)^\omega$ by Lemma \ref{C1}.  In particular, we have $vv^{\alpha_1} \geq v^{\alpha_1}v$.  Combining inequalities yields $vv^{\alpha_1} = v^{\alpha_1}v$, implying $v$ is a power of $v^{\alpha_1}$ by Corollary \ref{I}. Thus $$(v^{\alpha_1})^\omega=v^\omega\geq \ell_0^\omega\geq \ell_{m+1}^\omega=(v^{\alpha_1})^\omega,$$ so Corollary \ref{I} and Lemma \ref{C1} imply $\ell_0=v^{\alpha_1}=v$. This contradicts our choice of $m$, hence $|\ell_{k+1}| < |v|$ and $\ell_{k+1} \sim v$, as desired.
 
 In the other case, we need to consider is $|\ell_{k+1}| > |v|$ and $m \geq k$.  Let $\ell_{k+1} = v^r v^\alpha$ for $r \in \N$ and $\alpha \in (0,1)$, noting that $r + \alpha \not\in \N$ since $\ell_{k+1}$ is generalized Lyndon, and hence primitive.  If $k + 1 = n$ or if $|\ell_{k+2}| \geq v$, then $\ell_{k+1}v$ is a factor of $w$.  Note by our inductive hypothesis that $v^rvv^\alpha$ is a prefix of $w$.  By the generalized Lyndon property, we have $v^rv^\alpha v \geq v^r vv^\alpha$.  However, we also have 
 $$v^rv^\alpha v \sim (v^rv^\alpha)^\omega = \ell_{k+1}^\omega \leq \ell_0^\omega \leq v^\omega \sim v^rvv^\alpha.$$
 Combining inequalities yields $v^rv^\alpha v = v^rv v^\alpha$, implying $v$ is a power of $v^{\alpha}$ by Corollary \ref{I}. This means $\ell_{k+1}=v^rv^\alpha$ is a power of $v^\alpha$, contradicting the primitiveness of $\ell_{k+1}$, so we must have $k < n-1$ and $|\ell_{k+2}| < |v|$.  Since we assume $|u| > |v|$, there must exist some $q \in \{k+2,\dots,n\}$ such that $|\ell_q| < n$ and $\ell_q\ell_{q+1}\cdots \ell_n u \sim \ell_q v$. Notably, the largest value of $q$ such that $|\ell_q|<n$ works.  Then $v\ell_q$ is a prefix of $w$, and $\ell_q v$ is a factor of $w$. By the generalized Lyndon property of $w$ and our inductive hypothesis, we have $\ell_q v \geq v\ell_q$.  However, we also have $\ell_q^\omega \leq \ell_0^\omega \leq  v^\omega$, hence $(\ell_qv)^\omega \leq v^\omega$ by Lemma $\ref{C1}$.  In particular, $\ell_qv \lesssim vv \sim v\ell_q$.  Again, we combine inequalities and use Corollary \ref{I} and Lemma \ref{C1} to conclude $\ell_q = v$, our final contradiction. 
\end{proof}

\begin{lemma}\label{S}
If $w$ satisfies the hypotheses of Lemma \ref{U}, then $\ell_1 = \cdots = \ell_n = v$.
\end{lemma}
\begin{proof}
It is enough to show that $\ell_n = v$, since $v^\omega \geq \ell_0^\omega \geq \ell_i^\omega \geq \ell_n^\omega= v^\omega$.  If $\ell_n \not= v$, then by Lemma \ref{Q}, we have $\ell_n = v^\alpha$ for some $\alpha \in (0,1)$.  Moreover, by Lemma \ref{U} and our assumption $|u| \geq |v|$, we have $\ell_1\dots\ell_nu \sim v$, hence $w = v\ell_1\dots\ell_nu \sim vv$.  Thus, $v^\alpha v$ is a factor of $w$ and $vv^\alpha$ is a prefix, so the generalized Lyndon property of $w$ implies $vv^\alpha \leq v^\alpha v$.  Since we have $(v^\alpha)^\omega = \ell_n^\omega \leq \ell_0^\omega \leq v^\omega \sim v \lesssim v^\alpha v,$
Lemma \ref{J} implies $v \sim (v^\alpha)^\omega$.  In particular, $v^\alpha v \sim (v^\alpha)^\omega  \leq v^\omega \sim v v^\alpha$.  Combining inequalities, we have $vv^\alpha \leq v^\alpha v$.  This implies $v$ is a power of $v^\alpha$ by Corollary \ref{I}.  Thus $(v^\alpha)^\omega = v^\omega \geq \ell_0^\omega \geq \ell_n^\omega = (v^\alpha)^\omega$. In particular $\ell_0^\omega = v^\omega$, which implies $v = \ell_0$ by the generalized Lyndon property of $\ell_0$.  Therefore $v = \ell_0 = v^\alpha$, contradicting our choice of $\alpha$.
\end{proof}

\begin{corollary}\label{T}
Let $w$ be as in the statement of Lemma \ref{S}.  If we additionally assume that $v$ is a proper suffix of $\ell_0$ and $n \geq 1$, then $|u| < |v|$.
\end{corollary}
\begin{proof}
    By Lemma \ref{S}, we have $\ell_1 = \ell_2 = \cdots = \ell_n = v$.  Since $v$ is a proper suffix of $\ell_0$, we have $v^\omega > \ell_0^\omega$.  Thus $v^\omega > \ell_0^\omega \geq \ell_1^\omega = v^\omega$, which is a contradiction.
\end{proof}

\begin{theorem}\label{R}
 An aperiodic infinite word with no generalized Lyndon suffix has a unique nonincreasing factorization into finite generalized Lyndon words.
\end{theorem}
\begin{proof}
Suppose $w$ is an aperiodic word with no generalized Lyndon suffix such that $w$ has two distinct nonincreasing factorizations into generalized Lyndon words.  Note that each factor in both factorizations must be finite.  We can remove any initial common factors, so without loss of generality $w = \prod_{i=0}^\infty w_i = \prod_{j=0}^\infty \ell_j$ where $w_i^\omega \geq w_{i+1}^\omega$ for all $i \in \Z_{\geq 0}$, $\ell_j^\omega \geq \ell_{j+1}^\omega$ for all $j \in \Z_{\geq 0}$, and $|w_0| >  |\ell_0|$. Since we know finite words have unique nonincreasing generalized Lyndon factorizations from Theorem \ref{C3}, we have $\prod_{i=0}^x w_i \not= \prod_{j=0}^y \ell_j $ for any $x,y \in \Z_{\geq 0}$.  
\vspace{-0.3in}
\begin{figure}[h]
\centering
\includegraphics[width = 4.5 in]{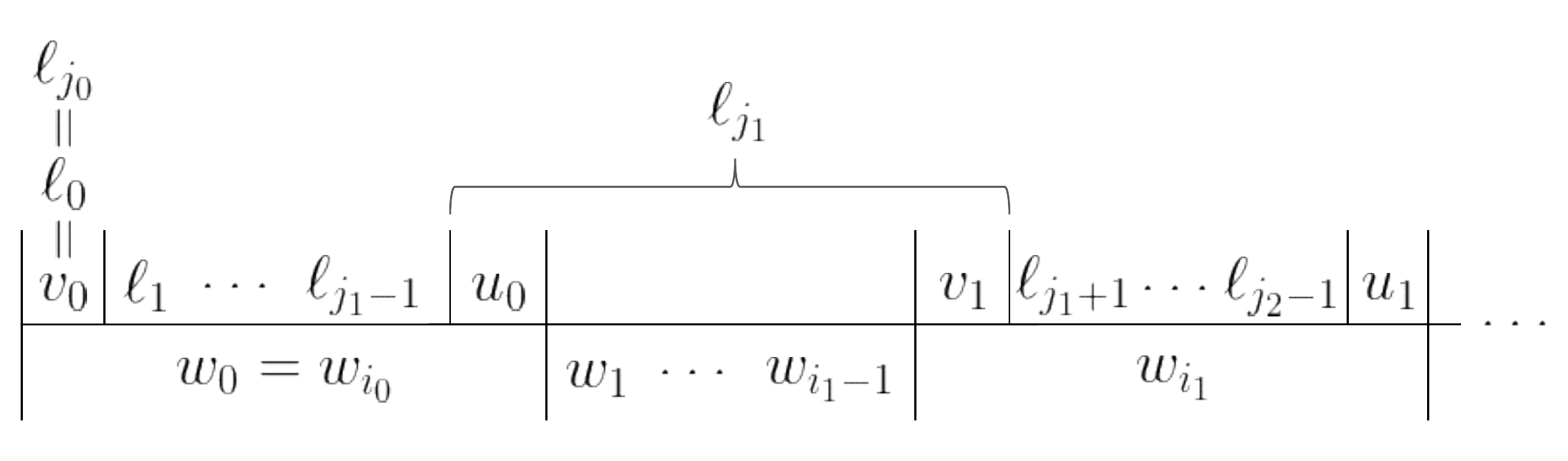}
\caption[Figure 1.]{The construction of $v_k$, $\ell_{j_k}$, $w_{i_k}$, and $u_k$}
\label{fig1}
\end{figure}
\vspace{-.2in}

Define $v_0 = \ell_0$, $\ell_{j_0} = \ell_0$, and $w_{i_0} = w_0$. We define $j_{k+1}$ to be the unique integer such that $w_{i_k}$ can be written as $v_k \ell_{j_k + 1} \cdots \ell_{j_{k+1} - 1} u_k$, where $u_k$ is a prefix of $\ell_{j_{k+1}}$.  We define $i_{k+1}$ to be the unique integer such that $\ell_{j_{k+1}}$  can be written as $u_k w_{j_k + 1}\cdots w_{j_{k+1}-1} v_{k+1}$, where $v_{k+1}$ is a prefix of $w_{j_{k+1}}$.  This construction is illustrated in Figure \ref{fig1}.  Observe that for each $k \in \Z_{\geq 0}$ we have that $v_k$ is a proper prefix of $w_{i_k}$, $u_k$ is a proper suffix of $w_{i_k}$, $u_k$ is a proper prefix of $\ell_{j_{k+1}}$, and $v_{k+1}$ is a proper suffix of $\ell_{j_{k+1}}$.

We aim to show that $|v_{k+1}| < |v_k|$ for each $k \in \Z_{\geq 0}$, and since this reduction can only be applied finitely many times, we will reach a contradiction.  Assume not, that $|v_{k+1}| \geq |v_k|$ for a certain $k \in \Z_{\geq 0}$.  
 
 First suppose that $|u_k| \geq |v_k|$, and note that we cannot have $u_k$ be a power of $v_k$  or $u_k = v_k$, or else $w_{i_k}$ is not primitive by Lemma \ref{S}. Thus Lemmas \ref{S} and \ref{U} imply that $u_k = v_k^r v_k^\alpha$ for some $r \in \N$ and $\alpha \in (0,1)$.   Moreover, Corollary \ref{T} implies that $w_{i_k} = v_k u_k$.   Lemma \ref{U} implies that $v_k \sim u_k$ and $u_k \sim v_{k+1}$, hence $|v_k|\leq |u_k|, |v_{k+1}|$ implies $v_k \sim v_{k+1}$. Furthermore, Lemma \ref{U} yields  $v_k \sim v_{k+1} \sim w_{i_{k}+1}\cdots w_{i_{k+1}-1}v_{k+1}$, so $u_kw_{i_{k}+1}\cdots w_{i_{k+1}-1}v_{k+1} \sim v_k^rv_k^\alpha v_k$.  Thus
  $$v_k^rv_kv_k^\alpha \sim v_k^\omega \geq \ell_{j_k}^\omega = (u_kw_{i_k+1} \cdots w_{i_{k+1}-1}v_{k+1})^\omega \sim v_k^rv_k^\alpha v_k.$$
  However, by the generalized Lyndon property of $w_{i_k}$ and Lemma \ref{Q}, we also have
$$v_k^rv_k^\alpha v_k \sim (v_k^r v_k^\alpha)^\omega =  u_k^\omega \geq (v_k\ell_{j_k + 1}\dots\ell_{j_{k+1} -1})^\omega = v_k^\omega \sim v_k^r v_k v_k^\alpha.$$
Combining inequalities yields $v_k^rv_k^\alpha v_k = v_k^r v_k v_k^\alpha$, which implies that $v_k$ and $v_k^\alpha$ are powers of a common word by Lemma \ref{C2}.  Thus $w_{i_k}$ is not primitive, contradicting that it is generalized Lyndon.  So in this case we have $|v_{k+1}| < |v_k|$.
  
Now suppose $|u_k| < |v_k|$.  By the generalized Lyndon property of $w_k$, we have $v_k \lesssim u_k$.  The nonincreasing property of our factors implies $u_k \sim \ell_{j_{k+1}}^\omega \leq \ell_{j_k}^\omega \leq v_k^\omega \sim v_k.$
Hence $v_k \sim u_k$, so $u_k = v_k^\alpha$ for some $\alpha \in (0,1)$.  Since $|u_k| < |v_k| \leq |v_{k+1}|$, Corollary \ref{T} applies to $\ell_{j_k}$.  In particular, $\ell_{j_{k+1}} = u_kv_{k+1}$.  
By the generalized Lyndon property of $w_{i_k}$ and $\ell_{j_{k+1}}$ along with Lemma \ref{C1}, we have
$$(v_k^\alpha)^\omega = u_k^\omega \geq w_{i_k}^\omega \geq w_{i_{k+1}}^\omega \sim v_{k+1} \sim v_{k+1}^\omega \geq u_k^\omega = (v_k^\alpha)^\omega.$$
Hence $v_{k+1} \sim (v_k^\alpha)^\omega$.  Note that $v_k \sim w_{i_k}^\omega$, so we also have $v_k \sim v_{k+1} \sim (v_k^\alpha)^\omega$.  On the one hand, we have
$$v_k^\alpha v_k \sim u_kv_{k+1} = \ell_{j_{k+1}} \sim \ell_{j_{k+1}}^\omega \leq \ell_{j_k}^\omega \leq v_k^\omega \sim v_kv_k^\alpha.$$
However, the generalized Lyndon property of $w_1$ and Lemma \ref{U} imply
$$v_k^\alpha v_k = u_kv_k \sim u_kv_k \ell_{j_k + 1}\cdots \ell_{j_{k+1} - 1} \gtrsim v_k \ell_{j_k + 1}\cdots \ell_{j_{k+1} - 1} u_k \sim v_k u_k = v_k^\alpha v_k.$$
Therefore $v_k^\alpha v_k =  v_kv_k^\alpha$, which implies $v_k$ and $v_k^\alpha$ are powers of a common word.  However, we reach our final contradiction by noting that $\ell_{j_{k+1}}$ is not primitive, contradicting that it is generalized Lyndon.
\end{proof}

\begin{theorem}
Every infinite word has a unique factorization into a nonincreasing product of generalized Lyndon words.
\end{theorem}
\begin{proof}
This follows directly from Theorems \ref{L}, \ref{O}, and \ref{R}.
\end{proof}

\section{Characterization of Infinite Generalized Lyndon Words}\label{6}
Siromoney et al. showed in \cite{SMDS} that the infinite Lyndon words are precisely the limits of prefix-preserving increasing sequences of finite Lyndon words.  We show that this result still holds when Lyndon words are replaced with generalized Lyndon words provided that the infinite word is primitive.  

\begin{theorem}
A primitive infinite word is generalized Lyndon if and only if it has infinitely many generalized Lyndon prefixes.
\end{theorem}
\begin{proof}
Lemma \ref{A} handles the reverse direction. Suppose that there exists an infinite generalized Lyndon word $w$ with finitely many generalized Lyndon prefixes.  Since $w$ has infinitely many prefixes, one of its generalized Lyndon prefixes must appear in the unique nonincreasing generalized Lyndon factorizations (from Theorem \ref{C3}) of infinitely many of the prefixes of $w$.  

We will now use the method presented in the proof of Theorem \ref{N} to construct a nontrivial factorization of $w$, contradicting the result of Theorem \ref{L}.   Suppose that $w = \ell_1 \dots \ell_n w_n$ where $\ell_j$ is a finite generalized Lyndon word and $\ell_1^\omega \geq \ell_2^\omega\geq \ell_{n}^\omega>w_n$. Further suppose that $w$ has infinitely many prefixes with factorizations beginning with $\ell_1, \dots, \ell_n$.  If $w_n$ is not generalized Lyndon, the process proceeds as in Theorem \ref{N}.  

Suppose $w_n$ is generalized Lyndon.  If we can choose a generalized Lyndon prefix $\ell_{n+1}$ of $w_n$ such that infinitely many prefixes of $w$ have factorizations beginning with $\ell_1, \dots, \ell_n, \ell_{n+1}$, then the process can continue.  Otherwise, there must be infinitely many prefixes $p$ of $w_n$ such that $\ell_1, \dots, \ell_n, p$ is a factorization of a prefix of $w$.  In particular, we have that $p^\omega \leq \ell_n^\omega$ for infinitely many prefixes of $p$.  Taking the limit of these prefixes, we find that $w_n \leq \ell_n^\omega$.  Thus, $\ell_1, \dots, \ell_n, w_n$ is a nontrivial factorization of $w$, contradicting Theorem $\ref{L}$.

Therefore either the process terminates and produces a nontrivial finite generalized Lyndon factorization of $w$, or it continues indefinitely and produces a nonincreasing generalized Lyndon factorization of $w$.  Either case contradicts Theorem \ref{L}, so $w$ must have infinitely many generalized Lyndon prefixes.
\end{proof}

We cannot hope this result extends to the case where the infinite word is not primitive. For example, consider $(01)^\omega$ under the alternating order.  It has infinitely many Galois prefixes, namely the prefixes of the form $(01)^k 0$ for any $k \in \N$, but $01$ is not Galois.

\section{Further Directions}

Theorem \ref{S} shows that infinite generalized Lyndon words have infinitely many generalized Lyndon prefixes.  But which finite generalized Lyndon words can arise as Lyndon prefixes?  It is straightforward to see that if the alphabet $A$ is finite, then the maximum $1$-letter word will not arise as a prefix of any infinite generalized Lyndon word. However, we conjecture that every other finite generalized Lyndon word is extendable to an infinite generalized Lyndon word.

\begin{conjecture}
Every finite generalized Lyndon word of length at least $2$ is a prefix of an infinite generalized Lyndon word.
\end{conjecture}

Observe that every finite Lyndon word can be extended to an infinite Lyndon word by appending an $\omega$-power of the maximum letter.  However, in the generalized case there is no notion of a maximal letter appearing in a word.  Moreover, there exist finite generalized Lyndon words which cannot be extended to infinite generalized Lyndon words by appending a power of a letter.  For example, $01000010$ is a Galois word but $01000010(0^{\omega})$ and $01000010(1^\omega)$ are not Galois.  The Galois word $01000010$ is still extendable to an infinite Galois word by appending a slightly more complicated suffix, e.g., $01^\omega$.  Given an infinite word $w$, it may be interesting to characterize which finite generalized Lyndon words are extendable to an infinite generalized Lyndon word by appending $w$. 

Given that every word has a unique nonincreasing factorization into generalized Lyndon words, one may wish to characterize or compute this factorization. For example, given a simple representation (e.g. a finite expression of products and powers) of an infinite word and a generalized lexicographical ordering, one may wish to compute the factorization of the word in polynomial time.

In a different direction, the existence and uniqueness of a factorization of a general transfinite (ordinally indexed) word into Lyndon words is proved in \cite{BC}. It remains to be seen whether this factorization theorem still holds when using generalized Lyndon words. Lastly, one may seek a general characterization of the first factor in a generalized Lyndon factorization along the lines of \cite{Ufn}. While simple characterizations such as longest Lyndon prefix fail, there may be a more clever characterization lurking in the background.

\section{Acknowledgements}
We extend our thanks to the anonymous referees for their detailed comments.  The paper was greatly improved by their suggestions. 


\end{document}